\documentclass[11pt]{amsart}
\usepackage[notref,notcite]{showkeys}
\usepackage{graphicx}
\usepackage{color}
\usepackage{pdfsync}
\usepackage{amsfonts,amssymb} 
\newtheorem{thm}{Theorem}[section]
\newtheorem{corollary}[thm]{Corollary}
\newtheorem{lemma}[thm]{Lemma}

\theoremstyle{definition}
\newtheorem{definition}[thm]{Definition}

\theoremstyle{remark}
\newtheorem{remark}[thm]{Remark}
\numberwithin{equation}{section}

\newcommand*\diff{\mathop{}\mathrm{d}}

\newcommand\restr[2]{{
  \left.\kern-\nulldelimiterspace 
  #1 
  \vphantom{\big|} 
  \right|_{#2} 
  }}

\begin{document}

\title[Stochastic extended KdV equation]
{Existence of mild solution to stochastic extended Korteweg - de Vries equation}

\author[Karczewska]{Anna Karczewska}
\address{Faculty of Mathematics, Computer Science and Econometrics\\ University of Zielona G\'ora, Szafrana 4a, 65-516 Zielona G\'ora, Poland}
 \email{a.karczewska@wmie.uz.zgora.pl} \thanks{}

\author[Szczeci\'nski]{Maciej Szczeci\'nski}
\address{Faculty of Mathematics, Computer Science and Econometrics\\ University of Zielona G\'ora, Szafrana 4a, 65-516 Zielona G\'ora, Poland}
 \email{m.szczecinski@wmie.uz.zgora.pl} \thanks{}

\date{\today}

\subjclass[2010]{93B05, 93C25, 45D05, 47H08, 47H10}

\keywords{Extended KdV equation, mild solution}


\begin{abstract}
In the paper we consider stochastic Korteweg - de Vries - type equation. We give sufficient conditions for the existence and uniqueness of local mild solution to the equation with additive noise. 
We discuss possibility of globalization of mild solution, as well.
\end{abstract}

\maketitle

\section{Introduction} \label{intro}

Nonlinear wave equations 
 attracted an enormous attention in many fields, e.g.\ physics (hydrodynamics, plasma physics, optics), technology (electric circuits, light impulses propagation) and  biology (neuroscience models, protein and DNA motion).
Usually such equations are obtained as a kind of approximation and/or simplification of the set of several more fundamental equations governing the system with their boundary and initial conditions. Approximations are usually based on perturbative approach in which some small parameters, related to particular properties of the considered system, appear.
Then the relevant quantities are expanded in power series of these small parameters. Limitation to terms of the first or second order allows to derive approximate nonlinear wave equations describing the evolution of a given system.

In several fields the lowest (first) order equation takes form of the Korteveg - de Vries equation (commonly denoted as KdV) Korteweg and de Vries \cite{KdV}
\begin{equation}\label{kdv0}
\frac{\partial u}{\partial t} + 6 u \frac{\partial u}{\partial x} + \frac{\partial^3 u}{\partial x^3}=0.
\end{equation}
It was derived firstly for surface gravity waves on shallow water but later found in many other systems, see, e.g.\ Drazin and Johnson \cite{DrazJohn}, Infeld \cite{EIGR}, Jeffrey \cite{Jeffrey}, Remoissenet \cite{Rem}.

Although KdV equation displays dominant features of weekly dispersive nonlinear waves, it is a valid approximation only for constant water depth. For the case of uneven bottom or when surface tension becomes important perturbative approach to Euler equations should be applied up to second order in small parameters. Then linear terms with fifth order derivatives 
 and new nonlinear terms appear in final nonlinear wave equations. Examples of such equations, called {\tt extended KdV} or {\tt KdV2} can be found in many papers, see, e.g.\ Karczewska et.al. \cite{KRR}, Karczewska et.al. \cite{KRI} and references therein. 
Nonlinear dispersive waves attracted a great attention of mathematicians. Among many examples of mathematical description of those problems we point out books of Linares and Ponce \cite{LiPo} and Tao \cite{Tao}. 

Surface water waves are subjected to some unpredictable influences of the environment, like winds, bottom fluctations, etc. These unknown factors can be accounted for by introducing a forcing term of stochasic nature into wave equation.

In the current paper we study stochastic version of KdV2 -- type equation derived in Karczewska et.al. \cite{KRR}, Karczewska et.al. \cite{KRI}. 
We supply sufficient conditions for the existence and uniqueness of local mild solution to the Korteweg - de Vries type equation of the form (\ref{NL1}) below.  We follow and generalize the approach of de Bouard and Debussche \cite{Deb} and Kenig, Ponce and Vega \cite{KPV91,KPV93} to such equation.

We obtained the existence and uniqueness results on random interval. The generalization of these results to any time interval with the approach due to de Bouard and Debussche \cite{Deb}  is not possible since they use some properties of classical KdV equation and its invariants. In our case, for extended KdV equation, there exists only one (the lowest) exact invariant, the other ones are only adiabatic (approximate) \cite{KRIad}.

In Section \ref{Snit} we discuss  possibility for some kind of globalization of obtained mild solution to stochastic extended KdV equation studied. We use the near-identity transformation (NIT for short) Kodama \cite{Kodama}, Dullin et.al. \cite{Dull2001} in order to transform original non-integrable extended KdV equation into asymptotically equivalent equation which has Hamiltonian form and therefore is integrable. The term {\tt asymptotic equivalence} means that solutions of both equations coincide when physically relevant coefficients of the equations tend to zero (for details, see Section \ref{Snit}).

\section{Existence and uniqueness} \label{exist}  

In this section we prove the existence and uniqueness of mild solution 
on a random interval to stochastic extended KdV-type equation of the form
\begin{equation}\label{NL1}
 \begin{aligned}
		 \diff u + \left( \frac{\partial ^{3} u }{\partial x^{3}} + u \frac{\partial u}{\partial x} + u \frac{\partial^{3} u}{\partial x^{3}} + \frac{\partial u}{\partial x}\frac{\partial^{2} u}{\partial x^{2}} \right) \diff t = \Phi \diff W , \quad x\in\mathbb{R}, \quad t\ge 0.
	\end{aligned}
\end{equation}
Motivation for studying the equation (\ref{NL1}) is given in Section \ref{Snit}. 
In (\ref{NL1}), $W$ is a cylindrical Wiener process defined on the stochastic basis $(\Omega, \mathcal{F}, (\mathcal{F}_t)_{t \geq 0}, \mathbb{P})$ with values on $L^2(\mathbb{R})$ adapted to the filtration 
$(\mathcal{F}_{t})_{t\ge 0}$. 
The operator  ~$\Phi$ belongs to $L_2^0$, where
$L_2^0 :=L_2^0(L^2(\mathbb{R});H^{\sigma}(\mathbb{R}))$~ is the space of Hilbert-Schmidt operators acting from $L^2(\mathbb{R})$ into $H^{\sigma}(\mathbb{R})$~ and  
$H^{\sigma}(\mathbb{R})$~ is the Sobolev space (see, e.g., Adams \cite{Adams}), ~$\sigma >0$.
The equation (\ref{NL1}) is supplemented with an initial condition 
\begin{equation} \label{NL1WP}
u(x,0) = u_0(x), \quad x\in\mathbb{R}, \quad t\ge 0.
\end{equation}

\begin{definition} \label{3.1}
A stochastic process ~$u(t),~ t\ge 0$, defined on the basis $(\Omega, \mathcal{F}, (\mathcal{F}_t)_{t \geq 0}, \mathbb{P})$ is said to be a {\tt mild solution} to (\ref{NL1})-(\ref{NL1WP}), if 	
\begin{equation}\label{mNL1} 
	u(t) = V(t)u_{0} + \int_{0}^{t} V(t-s) \left(  u \frac{\partial u}{\partial x} + u \frac{\partial^{3} u}{\partial x^{3}} + \frac{\partial u}{\partial x}\frac{\partial^{2} u}{\partial x^{2}} \right) \diff s + \int_{0}^{t} V(t-s)\, \Phi \diff W(s).
\end{equation}
\end{definition}
In (\ref{mNL1}), ~$V(t), t\ge 0$, is a unitary group generated by the linear part of the KdV equation (\ref{kdv0}).

To simplify notation we will use the following abbreviation for stochastic convolution
\begin{equation} \label{W_V}
W_{V}(t):= \int_0^t V(t-s)\,\Phi \,dW(s),\quad t\ge 0.
\end{equation}

\begin{definition}\index{$_{k}A$}
For a given set $A$~ by ~$_{n}A$ we shall denote the biggest subset of $A$ defined as
\begin{equation}\nonumber
_{n}A := \left\{u\in A: \frac{\partial^{kn}u}{\partial x^{kn}} \in A ,k\in \mathbb{N}\right\}. 
\end{equation}
\end{definition}

In the paper we shall use the following notation.
\begin{align*}
 X_{\sigma} (T):= \left\{~ \right.& u\in L^{\infty}(0,T;H^{\sigma}(\mathbb{R}))\cap L^2(\mathbb{R};L^{\infty}([0,T])), \hspace{5ex} \\
 &\hspace{5ex} \left. D^{\sigma} \partial_x u \in  L^{\infty}(\mathbb{R},L^2([0,T])), \partial_x u \in
L^4([0,T];L^{\infty}(\mathbb{R})) \right\}
\end{align*}

\begin{lemma}\label{LEM2}
If $(u_{0})_{2x}\in H^{\sigma}(\mathbb{R})$, then $V(t)(u_{0})_{2x} \in X_{\sigma}(T)$,~ ~$\sigma > \frac{3}{4}$.
\end{lemma} 
\vspace{-2ex}
\begin{proof}
Proof comes  from Proposition 3.5 in de Bouard and Debussche \cite{Deb}.
\end{proof}
 \vspace{-1ex}
Now, we can formulate first result.

\begin{thm} \label{aux}
Assume that $\Phi\in L_2^0(L^2(\mathbb{R});H^\sigma(\mathbb{R}))$ 
with $\alpha >\frac{3}{4}$ and 
\begin{equation}\label{cs1}
\frac{\partial^{2}}{\partial x^{2}} \left[\int_{0}^{t} V(t-s)\Phi \diff W(s)\right] \in L^{2}\left(\Omega;L_{x}^{2}\left(L_{t}^{\infty}\right)\right).
\end{equation}
Then $\frac{\partial^{2}}{\partial x^{2}}W_{V} \in \widehat{X}_{\sigma}(T)$, $\mathbb{P}$-almost surely, where
\begin{equation}\nonumber
\begin{aligned}
\widehat{X}_{\sigma}(T) :=& \left\{ u: L^{2} (\mathbb{R};L^{\infty}([0,T])), \quad D^{\sigma} \partial_{x} u \in L^{\infty} (\mathbb{R}, L^{2}([0,T])), \right. \\
 &\left. \partial_{x} u \in L^{4} ([0,T];L^{\infty}(\mathbb{R})) \right\} ,
\end{aligned}
\end{equation}
for any $T>0$ and all $\sigma$, such that $\frac{3}{4}<\sigma<1$.
\end{thm}
 \vspace{-2ex}
\begin{proof}
For reader's convenience the proof of Theorem \ref{aux} is postponed to the  section \ref{proofs3}.
\end{proof} 

\begin{corollary}
Assume that $\Phi\in L_2^0(L^2(\mathbb{R});H^\sigma(\mathbb{R}))$ 
and for $\sigma>\frac{3}{4}$ holds
\begin{equation}\nonumber
\frac{\partial^{2}}{\partial x^{2}}W_{V}\in L^{2}\left(\Omega;L_{t}^{\infty}\left(H^{\sigma}_{x}\right)\right).
\end{equation}
Then $\frac{\partial^{2}}{\partial x^{2}}W_{V} \in X_{\sigma}(T)$, $\mathbb{P}$-almost surely for any $T >0$, and  $\sigma$ such that $\frac{3}{4}<\sigma<1$.
\end{corollary}

Now, we are able to formulate the existence and uniqueness result.

\begin{thm} \label{t31}
Assume that $u_{0} \in {_{2}L^{2}} \left( \Omega; H^{1}(\mathbb{R}) \right) \cap  {_{2}L^{4}} \left( \Omega ;L^{2}(\mathbb{R}) \right)$ and it is  $\mathcal{F}_{0}$-measurable and   $\Phi \in L_2^0 \left( L^{2} (\mathbb{R}) ; H^{1} (\mathbb{R}) \right)$. If 
(\ref{cs1}) holds 
then there exists a unique mild solution to the equation (\ref{NL1}) with initial condition
(\ref{NL1WP}), such that ~$u \in {_{2}X_{\sigma}} (T)$ 
 almost surely for some ~$T>0$ and for any $\sigma \in \left(\frac{3}{4},1\right)$. 
\end{thm}
 
\begin{proof} \label{p31}
As we have already written, in the proof we follow the method used in  de Bouard and Debussche \cite{Deb}. 

We introduce the mapping $\mathcal{T}$ defined as follows 
\begin{equation} \label{defT}
\mathcal{T}(u) := V(t)u_{0} + \int_{0}^{t} V(t-\tau) (u \frac{\partial u}{\partial x} + u \frac{\partial^{3} u}{\partial x^{3}} + \frac{\partial u}{\partial x}\frac{\partial^{2} u}{\partial x^{2}}) \diff \tau +  W_V(t), \quad t\ge 0. 
\end{equation}
Then
\begin{eqnarray} \label{dT1}
\mathcal{T}(u) 
& \!\! =V(t)u_{0} \!\! &+ \int_{0}^{t} V(t-\tau) (u \partial_{x} u + u \partial_{x} u_{2x} + u_{2x} \partial_{x} u) \diff \tau + W_V(t)) = \nonumber \\
&\!\!=  V(t)u_{0} \!\! &+ \int_{0}^{t} V(t-\tau) (u \partial_{x} u) \diff \tau + \int_{0}^{t} V(t-\tau) (u \partial_{x} u_{2x}) \diff \tau \\
&  \!\! &+ \int_{0}^{t} V(t-\tau) (u_{2x} \partial_{x} u) \diff \tau + W_V(t), \quad t\ge 0. \nonumber
\end{eqnarray}

 We want to obtain the following condition
\begin{equation}\label{WSIEBIE}
\left[\; u\in \left\{u: u\in X_{\sigma}(T), u_{2x} \in X_{\sigma}(T)\right\}_{}^{} \;\right] \; \Longrightarrow \; \left[\;  \mathcal{T}(u) \in \left\{u: u\in X_{\sigma}(T), u_{2x} \in X_{\sigma}(T)\right\}_{}^{} \; \right].
\end{equation}

From Theorem 3.2 and Proposition 3.5 in  de Bouard and Debussche 
\cite{Deb} and because ~$u,u_{2x}\in X_{\sigma}(T)$, the mapping $\mathcal{T}$ maps $X_{\sigma}(T)$ into itself if $u_{0} \in H^{\sigma}(\mathbb{R})$. We will check when  $\frac{\partial ^{2}}{\partial x^{2}}\mathcal{T}(u)\in X_{\sigma}(T)$. We have 
\begin{eqnarray}\label{T1DIFF}
\frac{\partial ^{2}}{\partial x^{2}}\mathcal{T}(u) & = & \frac{\partial ^{2}}{\partial x^{2}} V(t)u_{0} + \frac{\partial ^{2}}{\partial x^{2}}\! \int_{0}^{t} \!V(t-\tau) (u u_{x}) \diff \tau + \frac{\partial ^{2}}{\partial x^{2}}\! \int_{0}^{t} \!V(t-\tau) (u u_{3x}) \diff \tau   \nonumber \\
&& \hspace{1ex}
+ \frac{\partial ^{2}}{\partial x^{2}}\! \int_{0}^{t}\! \!V(t-\tau) (u_{x} u_{2x}) \diff \tau \!+\! \frac{\partial ^{2}}{\partial x^{2}} \! \int_{0}^{t}\! \!V(t-\tau) \Phi\, dW(\tau)\\
&=& V(t)(u_{0})_{2x} \!+\! \! \int_{0}^{t}\! \!V(t-\tau) (uu_{3x} \!+\! 3u_{x}u_{2x}) \diff \tau + \! \int_{0}^{t}\! \!V(t-\tau) (u_{2x}u_{3x} + 2u_{x}u_{4x} + uu_{5x}) \diff \tau \nonumber \\
& &\hspace{1ex} + \! \int_{0}^{t}\! \!V(t-\tau) (3u_{2x}u_{3x} + u_{x}u_{4x}) \diff \tau +  \frac{\partial ^{2}}{\partial x^{2}}\! \int_{0}^{t}\! \!V(t-\tau) \Phi\, dW(\tau). \nonumber 
\end{eqnarray}


Let $u\in {} _{2}X_{\sigma}(T)$. Then $v=u_{2x}\in {} _{2}X_{\sigma}(T)$ and $v_{2x}\in {} _{2}X_{\sigma}(T)$. We have
\begin{equation}\label{RR1}
\begin{aligned}
\int_{0}^{t} V(t-\tau) (uu_{3x} + 3u_{x}u_{2x}) \diff \tau &= \int_{0}^{t} V(t-\tau) u\partial_{x}u_{2x}  \diff \tau + 3\int_{0}^{t} V(t-\tau) u_{2x}\partial_{x}u \diff \tau = \\
&= \int_{0}^{t} V(t-\tau) u\partial_{x}v  \diff \tau + 3\int_{0}^{t} V(t-\tau) v\partial_{x}u \diff \tau;
\end{aligned} 
\end{equation}
\begin{equation}\label{RR2}
\begin{aligned}
\int_{0}^{t} V(t-\tau) (u_{2x}u_{3x} + 2u_{x}u_{4x} + uu_{5x}) \diff \tau &= \int_{0}^{t} V(t-\tau) u_{2x}\partial_{x}u_{2x} \diff \tau + 2\int_{0}^{t} V(t-\tau) u_{4x}\partial_{x}u \diff \tau + \\
&+ \int_{0}^{t} V(t-\tau) u\partial_{x}u_{4x} \diff \tau \\
& = \int_{0}^{t} V(t-\tau) v\partial_{x}v \diff \tau + 2\int_{0}^{t} V(t-\tau) v_{2x}\partial_{x}u \diff \tau \\ & + \int_{0}^{t} V(t-\tau) u\partial_{x}v_{2x} \diff \tau;
\end{aligned}
\end{equation} 
\vspace{-3ex}
\begin{equation}\label{RR3}
\begin{aligned}
\int_{0}^{t} V(t-\tau) (3u_{2x}u_{3x} + u_{x}u_{4x}) \diff \tau &=  3\int_{0}^{t} V(t-\tau) u_{2x}\partial_{x}u_{2x} \diff \tau + \int_{0}^{t} V(t-\tau) u_{4x}\partial_{x}u \diff \tau  \\
&= 3\int_{0}^{t} V(t-\tau) v\partial_{x}v \diff \tau + \int_{0}^{t} V(t-\tau) v_{2x}\partial_{x}u \diff \tau.
\end{aligned}
\end{equation}

From Theorem 3.2, Proposition 3.5  de Bouard and Debussche \cite{Deb}, 
 Lemma \ref{LEM2} and Theorem \ref{aux} above, and equations (\ref{RR1})-(\ref{RR3}) we obtain that the mapping $\mathcal{T}$ maps the set $_{2}X_{\sigma}(T)$ into itself if $u_{0} \in {} _{2}H^{\sigma}(\mathbb{R})$ and $\Phi \in {} _{2}\mathcal{L}_{2} \left( L^{2}(\mathbb{R},H^{\sigma}(\mathbb{R}))\right)$. We want to find a ball $\mathcal{B}$ in $_{2}X_{\sigma}(T)$ centered at point $0$ and radius $2R$ such that the mapping 
 $\restr{\mathcal{T}}{\mathcal{B}}$ is contraction. More precisely, we want to have the following conditions 
\begin{equation}\label{NL1KON}
\begin{aligned}
\textrm{(i) }& \left|u\right|_{X_{\sigma}(T)} < 2R \Longrightarrow \left|\mathcal{T}(u)\right|_{X_{\sigma}(T)} <2R;\\
\textrm{(ii) }& \left|\mathcal{T}(u) - \mathcal{T}(v) \right|_{X_{\sigma}(T)} < \left| u-v \right|_{X_{\sigma}(T)}, \qquad \left|u\right|_{X_{\sigma}(T)}, \left|v\right|_{X_{\sigma}(T)} <2R.
\end{aligned}
\end{equation}

First, let us note that for any $u\in X_{\sigma}(T)$ there exists
 $M_{u}>0$ such that 
$$ 
\left|u_{2x}\right|_{X_{\sigma}(T)} = M_{u} \left|u\right|_{X_{\sigma}(T)}.
$$ 
Denote 
\begin{equation}\nonumber
M:=\sup \lbrace M_{u}: u\in X_{\sigma}(T) \rbrace.
\end{equation}
Then
\begin{equation}\label{sup}
\left|u_{2x}\right|_{X_{\sigma}(T)} \leq M \left|u\right|_{X_{\sigma}(T)}.\\
\end{equation}

From (\ref{sup}) and Proposition 3.5.  de Bouard and Debussche 
\cite{Deb} we obtain the following estimate 
\begin{equation}\nonumber
\begin{aligned}
\left|\mathcal{T}(u) \right|_{X_{\sigma}(T)} & \leq \; C_{1}(\sigma,T) |u_{0}|_{H^{\sigma}(\mathbb{R})} + C_{2}(\sigma,T) T^{\frac{1}{2}} |u|^{2}_{X_{\sigma}(T)} + C_{3}(\sigma,T) T^{\frac{1}{2}} |u|_{X_{\sigma}(T)}|u_{2x}|_{X_{\sigma}(T)}  \\
& \hspace{3ex} +  C_{4}(\sigma,T) T^{\frac{1}{2}} |u_{2x}|_{X_{\sigma}(T)}|u|_{X_{\sigma}(T)} 
+ |W_V|_{X_{\sigma}(T)}  \\
& \leq \; C_{1}(\sigma,T) |u_{0}|_{H^{\sigma}(\mathbb{R})} + C_{2}(\sigma,T) T^{\frac{1}{2}} |u|^{2}_{X_{\sigma}(T)} +  C_{3}(\sigma,T) T^{\frac{1}{2}} M |u|^{2}_{X_{\sigma}(T)} \\
& \hspace{3ex}+ C_{4}(\sigma,T) T^{\frac{1}{2}} M |u|^{2}_{X_{\sigma}(T)} + |W_V|_{X_{\sigma}(T)}.
\end{aligned}
\end{equation}
Here and below we write for shortening $W_V$ instead of $W_V(t), ~t\ge 0$.

 Since $C_{i}(\sigma,T)$, $i=1,2,3,4$, are nondecreasing with respect to $T$, we can use \linebreak 
$C(\sigma,T) := \underset{T}{\max} \lbrace C_{1}(\sigma,T), C_{2}(\sigma,T), C_{3}(\sigma,T), C_{4}(\sigma,T) \rbrace$, which is nondecreasing with respect to $T$ to our estimate. We obtain
\begin{equation}\nonumber
\begin{aligned}
\left|\mathcal{T}(u) \right|_{X_{\sigma}(T)} & \leq  C(\sigma,T) |u_{0}|_{H^{\sigma}(\mathbb{R})} + C(\sigma,T) T^{\frac{1}{2}} |u|^{2}_{X_{\sigma}(T)} + C(\sigma,T) T^{\frac{1}{2}} M |u|^{2}_{X_{\sigma}(T)}  \\
& \hspace{3ex} + C(\sigma,T) T^{\frac{1}{2}} M |u|^{2}_{X_{\sigma}(T)}
 + |W_V|_{X_{\sigma}(T)}\\
& = C(\sigma,T) |u_{0}|_{H^{\sigma}(\mathbb{R})} + C(\sigma,T) T^{\frac{1}{2}} |u|^{2}_{X_{\sigma}(T)} \left ( 1 + 2M \right) + |W_V|_{X_{\sigma}(T)}.
\end{aligned}
\end{equation}

Now, we shall find $R$ fulfilling condition (\ref{NL1KON})(i). Assume that 
$|u|_{X_{\sigma}(T)}<2R$. Then we have 
\begin{equation}\nonumber
\begin{aligned}
\left|\mathcal{T}(u) \right|_{X_{\sigma}(T)} \leq C(\sigma,T) |u_{0}|_{H^{\sigma}(\mathbb{R})} + C(\sigma,T) T^{\frac{1}{2}} 4R^{2} \left ( 1 + 2M \right) + |W_V|_{X_{\sigma}(T)}.
\end{aligned}
\end{equation}
We want to receive 
\begin{equation}\nonumber
C(\sigma,T) |u_{0}|_{H^{\sigma}(\mathbb{R})} + C(\sigma,T) T^{\frac{1}{2}} 4R^{2} \left ( 1 + 2M \right) + |W_V|_{X_{\sigma}(T)} \leq 2R.
\end{equation}
This is equivalent to
\begin{equation}\nonumber
C(\sigma,T) |u_{0}|_{H^{\sigma}(\mathbb{R})}  + |W_V|_{X_{\sigma}(T)} \leq 2R - C(\sigma,T) T^{\frac{1}{2}} 4R^{2} \left ( 1 + 2M \right).
\end{equation}
Let us note that it is enough to have such $R$ that
\begin{equation}\nonumber
C(\sigma,T) |u_{0}|_{H^{\sigma}(\mathbb{R})}  + |W_V|_{X_{\sigma}(T)}\; \leq \;R \;\leq \;2R - C(\sigma,T) T^{\frac{1}{2}} 4R^{2} \left ( 1 + 2M \right).
\end{equation}
From the second inequality we obtain 
\begin{equation}\nonumber
\begin{aligned}
R \leq\; &\; 2R - C(\sigma,T) T^{\frac{1}{2}} 4R^{2} \left ( 1 + 2M \right), \qquad \mbox{then} \\
0 \leq\; &\; R - C(\sigma,T) T^{\frac{1}{2}} 4R^{2} \left ( 1 + 2M \right) \\
\mbox{and} \qquad 0 \leq\; &\; R [1 - 4R C(\sigma,T) T^{\frac{1}{2}} \left ( 1 + 2M \right)], \\
\mbox{and} \qquad 0 \leq\; &\; 1 - 4R C(\sigma,T) T^{\frac{1}{2}} \left ( 1 + 2M \right), \\
\mbox{and~finally} \qquad 
1 \geq\; &\; 4R C(\sigma,T) T^{\frac{1}{2}} \left ( 1 + 2M \right).
\end{aligned}
\end{equation}

Hence, in order to obtain (\ref{NL1KON})~ (i), the following inequalities must  hold 
\begin{equation}\label{NL1KONi}
\begin{cases}
C(\sigma,T) |u_{0}|_{H^{\sigma}(\mathbb{R})}  + |W_V|_{X_{\sigma}(T)} & \leq  R \\
4R C(\sigma,T) T^{\frac{1}{2}} \left ( 1 + 2M \right) & \leq  1 .
\end{cases} 
\end{equation}

Let us note that the second condition in (\ref{NL1KONi}) will hold too, if 
\begin{equation}\label{KONi} \index{$\kappa$ - any constant $>1$}
\kappa  4R C(\sigma,T) T^{\frac{1}{2}} \left ( 1 + 2M \right)  \leq  1
\end{equation}
for any fixed constant $\kappa >1$.

 Now, we will check when the condition (\ref{NL1KON})(ii) holds. First, we shall estimate the norm 
$\left|\mathcal{T}(u) - \mathcal{T}(v) \right|_{X_{\sigma}(T)}$. We can write
\begin{equation}\nonumber
\begin{aligned}
\left|\mathcal{T}(u) \right. &- \left. \mathcal{T}(v) \right|_{X_{\sigma}(T)} =\\
 &=  \left| \int_{0}^{t} V(t-\tau) (u \partial_{x} u - v \partial_{x}v) \diff \tau + \int_{0}^{t} V(t-\tau) (u \partial_{x} u_{2x} - v \partial_{x} v_{2x}) \diff \tau  \right. \\
& \hspace{4ex} +  \left. \int_{0}^{t} V(t-\tau) (u_{2x} \partial_{x} u - v_{2x} \partial_{x}v) \diff \tau  \right|_{X_{\sigma}(T)} \\
&=  \frac{1}{2} \left| \int_{0}^{t} V(t-\tau) \left[ u \partial_{x} (u-v) + (u-v) \partial_{x} u + v \partial_{x} (u-v) + (u-v) \partial_{x} v \right] \diff \tau \right. \\
& \hspace{4ex} +  \int_{0}^{t} V(t-\tau) \left[ u \partial_{x} (u-v)_{2x} + (u-v) \partial_{x} u_{2x} + v \partial_{x} (u-v)_{2x} + (u-v) \partial_{x} v_{2x} \right] \diff \tau   \\
& \hspace{4ex} + \left. \int_{0}^{t} V(t-\tau) \left[ u_{2x} \partial_{x} (u-v) + (u-v)_{2x} \partial_{x} u + v_{2x} \partial_{x} (u-v) + (u-v)_{2x} \partial_{x} v \right] \diff \tau    \right|_{X_{\sigma}(T)}   \\
& \leq  \frac{1}{2} \left| \int_{0}^{t} V(t-\tau) \left[ u \partial_{x} (u-v) + (u-v) \partial_{x} u + v \partial_{x} (u-v) + (u-v) \partial_{x} v \right] \diff \tau \right|_{X_{\sigma}(T)}  \\
& \hspace{4ex} +  \frac{1}{2} \left|\int_{0}^{t} V(t-\tau) \left[ u \partial_{x} (u-v)_{2x} + (u-v) \partial_{x} u_{2x} + v \partial_{x} (u-v)_{2x} + (u-v) \partial_{x} v_{2x} \right] \diff \tau \right|_{X_{\sigma}(T)}  \\
& \hspace{4ex} +  \frac{1}{2} \left|\int_{0}^{t} V(t-\tau) \left[ u_{2x} \partial_{x} (u-v) + (u-v)_{2x} \partial_{x} u + v_{2x} \partial_{x} (u-v) + (u-v)_{2x} \partial_{x} v \right] \diff \tau \right|_{X_{\sigma}(T)}  .
\end{aligned}\end{equation}

Then
\begin{equation}\nonumber
\begin{aligned}
\left|\mathcal{T}(u) \right. &- \left. \mathcal{T}(v) \right|_{X_{\sigma}(T)} \\
&\leq  \frac{1}{2} C(\sigma,T)T^{\frac{1}{2}}\left| u \right|_{X_{\sigma}(T)} \left| u-v \right|_{X_{\sigma}(T)} + \frac{1}{2} C(\sigma,T)T^{\frac{1}{2}}\left| u-v \right|_{X_{\sigma}(T)} \left| u \right|_{X_{\sigma}(T)}  \\
& \hspace{4ex} + \frac{1}{2} C(\sigma,T)T^{\frac{1}{2}}\left| v \right|_{X_{\sigma}(T)} \left| u-v \right|_{X_{\sigma}(T)} + \frac{1}{2} C(\sigma,T)T^{\frac{1}{2}}\left| u-v \right|_{X_{\sigma}(T)} \left| v \right|_{X_{\sigma}(T)} \\
& \hspace{4ex} +  \frac{1}{2} C(\sigma,T)T^{\frac{1}{2}}\left| u \right|_{X_{\sigma}(T)} \left| (u-v)_{2x} \right|_{X_{\sigma}(T)} + \frac{1}{2} C(\sigma,T)T^{\frac{1}{2}}\left| u-v \right|_{X_{\sigma}(T)} \left| u_{2x} \right|_{X_{\sigma}(T)}  \\
& \hspace{4ex} + \frac{1}{2} C(\sigma,T)T^{\frac{1}{2}}\left| v \right|_{X_{\sigma}(T)} \left| (u-v)_{2x} \right|_{X_{\sigma}(T)} + \frac{1}{2} C(\sigma,T)T^{\frac{1}{2}}\left| u-v \right|_{X_{\sigma}(T)} \left| v_{2x} \right|_{X_{\sigma}(T)}  \\
& \hspace{4ex} +  \frac{1}{2} C(\sigma,T)T^{\frac{1}{2}}\left| u_{2x} \right|_{X_{\sigma}(T)} \left| u-v \right|_{X_{\sigma}(T)} + \frac{1}{2} C(\sigma,T)T^{\frac{1}{2}}\left| (u-v)_{2x} \right|_{X_{\sigma}(T)} \left| u \right|_{X_{\sigma}(T)}  \\
& \hspace{4ex} + \frac{1}{2} C(\sigma,T)T^{\frac{1}{2}}\left| v_{2x} \right|_{X_{\sigma}(T)} \left| u-v \right|_{X_{\sigma}(T)} + \frac{1}{2} C(\sigma,T)T^{\frac{1}{2}}\left| (u-v)_{2x} \right|_{X_{\sigma}(T)} \left| v \right|_{X_{\sigma}(T)} \\
&=  \frac{1}{2} C(\sigma,T)T^{\frac{1}{2}} \left| u-v \right|_{X_{\sigma}(T)} \left[ 2 \left| u \right|_{X_{\sigma}(T)} + 2 \left| v \right|_{X_{\sigma}(T)} + 2 \left| u_{2x} \right|_{X_{\sigma}(T)} + 2 \left| v_{2x} \right|_{X_{\sigma}(T)} \right] + \\
& \hspace{4ex} +  \frac{1}{2}  C(\sigma,T)T^{\frac{1}{2}} \left| (u-v)_{2x} \right|_{X_{\sigma}(T)} \left[ 2 \left| u \right|_{X_{\sigma}(T)} + 2 \left| v \right|_{X_{\sigma}(T)} \right]  \\
&\leq  C(\sigma,T)T^{\frac{1}{2}} \left| u-v \right|_{X_{\sigma}(T)} \left[  \left| u \right|_{X_{\sigma}(T)} +  \left| v \right|_{X_{\sigma}(T)} +  M \left| u \right|_{X_{\sigma}(T)} + M \left| v \right|_{X_{\sigma}(T)} \right]  \\
& \hspace{4ex}  +  C(\sigma,T)T^{\frac{1}{2}} M \left| u-v \right|_{X_{\sigma}(T)} \left[  \left| u \right|_{X_{\sigma}(T)} +  \left| v \right|_{X_{\sigma}(T)} \right].
\end{aligned}\end{equation} 

Finally we have 
\begin{equation}\nonumber\begin{aligned}
\left|\mathcal{T}(u) \right. &- \left. \mathcal{T}(v) \right|_{X_{\sigma}(T)} \\
& \le  C(\sigma,T)T^{\frac{1}{2}} \left| u-v \right|_{X_{\sigma}(T)} \left[  \left| u \right|_{X_{\sigma}(T)} +  \left| v \right|_{X_{\sigma}(T)} +  M \left| u \right|_{X_{\sigma}(T)} + M \left| v \right|_{X_{\sigma}(T)} \right]  \\
& \hspace{4ex}  +  C(\sigma,T)T^{\frac{1}{2}}  \left| u-v \right|_{X_{\sigma}(T)} \left[  M \left| u \right|_{X_{\sigma}(T)} +  M \left| v \right|_{X_{\sigma}(T)} \right]  \\
&=  C(\sigma,T)T^{\frac{1}{2}} \left| u-v \right|_{X_{\sigma}(T)} \left[ \left| u \right|_{X_{\sigma}(T)} (2M + 1) + \left| v \right|_{X_{\sigma}(T)} (2M + 1) \right]  \\
&=  C(\sigma,T)T^{\frac{1}{2}} \left| u-v \right|_{X_{\sigma}(T)} \left( \left| u \right|_{X_{\sigma}(T)} +  \left| v \right|_{X_{\sigma}(T)} \right) (2M + 1).
\end{aligned}
\end{equation}

Since  $\left| u \right|_{X_{\sigma}(T)} \leq 2R$ and $\left| v \right|_{X_{\sigma}(T)} \leq 2R$, we have
\begin{equation}\label{pom}
\left|\mathcal{T}(u) - \mathcal{T}(v) \right|_{X_{\sigma}(T)} \leq 4R C(\sigma,T)T^{\frac{1}{2}} \left| u-v \right|_{X_{\sigma}(T)}(2M + 1).
\end{equation}
From (\ref{KONi}) we know that 
\begin{equation}\nonumber
  4R C(\sigma,T) T^{\frac{1}{2}} \left ( 2M+1 \right)  \leq  \frac{1}{\kappa}, 
\end{equation}
so, putting this into  (\ref{pom}), we can write
\begin{equation}\nonumber
\left|\mathcal{T}(u) - \mathcal{T}(v) \right|_{X_{\sigma}(T)} \leq \frac{1}{\kappa} \left| u-v \right|_{X_{\sigma}(T)}. 
\end{equation}
Hence, the mapping $\restr{\mathcal{T}}{\mathcal{B}}$ is contraction if
 $\frac{1}{\kappa} \left| u-v \right|_{X_{\sigma}(T)} < \left| u-v \right|_{X_{\sigma}(T)}$, what is satisfied for any $\kappa > 1$.
So, we have to choose $R_{0}$ and  $T$ such that  
\begin{equation}\label{KONTRA}
\begin{cases}
C(\sigma,T) |u_{0}|_{H^{\sigma}(\mathbb{R})}  + |W_V|_{X_{\sigma}(T)} & \leq  R_{0} \\
\kappa 4R_{0} C(\sigma,T) T^{\frac{1}{2}} \left ( 1 + 2M \right) & \leq  1 \quad \textrm{for some constant~} \;\;\kappa > 1 .
\end{cases}
\end{equation}

\begin{remark}
In order to do this it is enough to take
\begin{equation}\nonumber
\begin{aligned}
M:=\sup \lbrace M_{u}: u\in X_{\sigma}(T), \left| u \right|_{X_{\sigma}(T)} \leq 4R \rbrace .\\
\end{aligned}
\end{equation}
\end{remark}
\begin{proof}
We estimated by $M$ only terms $\left| u \right|_{X_{\sigma}(T)}$, $\left| v \right|_{X_{\sigma}(T)}$ and $\left| u-v \right|_{X_{\sigma}(T)}$. Since $\left| u \right|_{X_{\sigma}(T)} \leq 2R$ and $\left| v \right|_{X_{\sigma}(T)} \leq 2R$, so $\left| u-v \right|_{X_{\sigma}(T)} \leq 4R$.
\end{proof}

 Hence, the mapping $\mathcal{T}$ maps the ball $\mathcal{B}$ in $_{2}X_{\sigma}(T)$ centered at $0$ with radius $2R$ into itself and, restricted to this ball, the mapping $\mathcal{T}$ is contraction. By Banach contraction theorem, the mapping $\mathcal{T}$ has fixed point in the set $_{2}X_{\sigma}(T)$, which is a unique solution to the equation (\ref{NL1}) with initial condition (\ref{NL1WP}). 
\end{proof}

\section{Near-identity transformation for KdV2}\label{Snit}

The famous Korteweg - de Vries equation Korteweg and de Vries \cite{KdV} was first obtained
in consideration of shallow water wave problem with ideal fluid model.
It is assumed that the fluid is inviscid and its motion is irrotational.
Then the set of hydrodynamic (Euler's) equations with appropriate boundary conditions at the flat bottom and unknown surface is obtained. Scaling transformation to dimensionless variables introduces small parameters that allow us to apply perturbation approach. First order perturbation approach
leads to KdV equation (below written in a fixed reference frame)
\begin{equation} \label{kdv}
\eta_t  +  \eta_x + \frac{3}{2} \alpha\,\eta\eta_x+ \frac{1}{6}\beta\, \eta_{3x} =0 .  
\end{equation}
More exact, second order perturbation approach gives the extended KdV Marchant and Smyth 
\cite{MS90} called also KdV2 Karczewska et.al. \cite{KRR}, Karczewska et.al.\cite{KRI}, Karczewska et.al.\cite{KRI2} which has the following form
\begin{equation} \label{kdv2}
\eta_t  +  \eta_x + \frac{3}{2} \alpha\,\eta\eta_x+ \frac{1}{6}\beta\, \eta_{3x} -\frac{3}{8}\alpha^2\eta^2\eta_x +
  \alpha\beta\,\left(\frac{23}{24}\eta_x\eta_{2x}+\frac{5}{12}\eta\eta_{3x} \right)+\frac{19}{360}\beta^2\eta_{5x} =0 .  
\end{equation}
In both equations (\ref{kdv}) and (\ref{kdv2}) there appear parameters 
 $\alpha,\beta$, which should be small. Parameter $\alpha:= \frac{A}{h}$
is the ratio of wave amplitude  $A$ to water depth $h$  and determines nonlinear terms. Parameter $\beta:=(\frac{h}{l})^2$, where $l$ is an  average wavelength  describes the dispersion properties. When $\alpha\approx\beta \ll 1$ we have a classical shallow water problem. 
However, our recent paper Infeld et.al. \cite{IKRR} showed that exact solutions of KdV2 (\ref{kdv2}) occur when $\frac{\beta}{\alpha}\lesssim 10^{-4}$. Therefore for further considerations we can safely neglect in (\ref{kdv2}) the last term with fifth derivative.

Transformation to a moving reference frame
\begin{equation} \label{ruch}
 x' =x-t \qquad \mbox{and} \qquad t'=t
\end{equation}
yields KdV2 equation in the form
\begin{equation} \label{kdv2'}
\eta_{t'}   + \frac{3}{2} \alpha\,\eta\eta_{x'}+ \frac{1}{6}\beta\, \eta_{3x'} -\frac{3}{8}\alpha^2\eta^2\eta_{x'} +
  \alpha\beta\,\left(\frac{23}{24}\eta_{x'}\eta_{2x'}+\frac{5}{12}\eta\eta_{3x'} \right) =0 .  
\end{equation}
In next steps we drop signs $'$ at $x'$ and $t'$, having in mind that (\ref{kdv2'}) represents the KdV2 in a moving frame.

 Kodama \cite{Kodama} showed that several nonlinear partial differential equations are {\tt asymptotically equivalent}. This term means that solutions to these equations converge to the same solution when parameters 
 $\alpha,\beta\to 0$. Kodama and several other authors Dullin et.al. 
 \cite{Dull2001}, Grimshaw \cite{Grimshaw}, Grimshaw et. al. \cite{GrimLect} have shown that asymptotically equivalent
equations are related to each other by near-identity transformation (NIT).

Let us introduce Near Identity Transformation (NIT for short) in the form used in  Dullin et.al. \cite{Dull2001}
\begin{equation} \label{nit}
\eta = \eta' \pm \alpha a \eta'^2 \pm \beta b \eta'_{xx}+\cdots
\end{equation}

[In the sequel we set the sign $+$. Then the inverse transformation, up to $O(\alpha^2)$ is \linebreak $\eta' = \eta -\alpha a\eta^2 -\beta b\eta_{xx}+\cdots$]

NIT preserves the structure of the equation
 (\ref{kdv2'}), at most altering some coefficients. Insertion (\ref{nit}) into (\ref{kdv2'}) gives (up to 2nd order in $\alpha,\beta$)
\begin{eqnarray} \label{3a}
\eta'_t + \eta'_x  \!\!&\!\!+\!\!& \!\!  \alpha \left[\left(\frac{3}{2}+2a\right)\eta'\eta'_x
+ 2 a \eta'\eta'_t\right] +\beta  \left[\left(\frac{1}{6}+b\right)\eta'_{3x}+b \eta'_{xxt} \right]  \\ \!\!&\!\!+\!\!& \!\! 
\alpha\beta \left\{\left[\left(\frac{23}{24}+a+ \frac{3}{2}b \right)\eta'_x\eta'_{2x} \right] + \left[\left(\frac{5}{12}+\frac{1}{3}a+ \frac{3}{2}b \right)\eta'\eta'_{3x} \right] \right\}
 \nonumber \\ \!\!&\!\!+\!\!& \!\!   
\alpha^2 \left(-\frac{3}{8}+\frac{9}{2} a\right)\eta'^2\eta'_x
+\beta^2  \frac{1}{6}b\eta'_{5x}  =0.\nonumber 
\end{eqnarray}

Since terms with derivatives with respect to $t$ appear with coefficients $\alpha$ and $\beta$, 
we can replace them by appropriate expressions obtained from (\ref{kdv2}) limited to first order (that is from KdV)
\begin{equation} \label{ept}
\eta'_t = -\eta'_x - \frac{3}{2}\alpha \eta'\eta'_x  -\frac{1}{6} \beta\eta'_{3x}
\end{equation}
and
\begin{equation} \label{epxxt}
\eta'_{xxt} = \partial_{xx}\left( -\eta'_x - \frac{3}{2}\alpha \eta'\eta'_x -\frac{1}{6} \beta\eta'_{3x}\right) = -\eta'_{3x}  - \frac{3}{2}\alpha (3\eta'_x\eta'_{2x}+ \eta'\eta'_{3x})   -\frac{1}{6} \beta\eta'_{5x}.
\end{equation}
Then terms (\ref{ept}) and (\ref{epxxt}) cause the following changes
\begin{eqnarray} \label{epop}
\alpha\,2a\eta'\eta'_t+ \beta\, b \eta'_{xxt}  \!\!&\!\!= \!\!&\!\! -2\alpha a\eta' \left( \eta'_x + \frac{3}{2}\alpha \eta'\eta'_x  +\frac{1}{6} \beta\eta'_{3x} \right) 
 -\beta b
\left[ \eta'_{3x}  + \frac{3}{2}\alpha (3\eta'_x\eta'_{2x}+ \eta'\eta'_{3x}) +\frac{1}{6} \beta\eta'_{5x}\right]
 \nonumber \\  \!\!& \!\!= \!\!& \!\!
 -2\alpha a \,\eta'\eta'_x -3\alpha^2 a \eta'^2\eta'_x  -\beta b \eta'_{3x} 
-\alpha\beta\left[ \frac{1}{2} b \,\eta'_x\eta'_{2x} + \left(\frac{1}{3} a +  \frac{3}{2} b\right) \eta'\eta'_{3x} \right] -\frac{1}{6} \beta^2 b \eta'_{5x}.
\end{eqnarray}

Insertion of (\ref{epop}) into (\ref{3a}) yields
\begin{eqnarray} \label{3b}
\eta'_t + \eta'_x  + \frac{3}{2} \alpha \eta'\eta'_x
 +\frac{1}{6}\beta \eta'_{3x}  &+& \alpha^2 \left(-\frac{3}{8}+\frac{3}{2} a\right)\eta'^2\eta'_x \nonumber
 \\ \!\!&\!\!+\!\!& \!\! 
\alpha\beta \left[\left(\frac{23}{24} +a- 3b \right)\eta'_x\eta'_{2x} +\frac{5}{12}\eta'\eta'_{3x} \right] 
=0.
\end{eqnarray}

Comparison of (\ref{3b}) with (\ref{kdv2'}) shows that only two coefficients are altered, that at the term containing 
 $\alpha^2$,  where $-\frac{3}{8}\rightarrow -\frac{3}{8}+\frac{3}{2}a$ and that with $\alpha\beta \eta'_x\eta'_{2x}$, where $\frac{23}{24}\rightarrow \frac{23}{24}+a-3b$. 

Equation (\ref{3b}) is asymptotically equivalent to (\ref{kdv2'}). NIT gives us some freadom in choosing coefficients $a,b$. They can be chosen such that the most nonlinear term (with 3-rd order nonlinearity) is canceled and the final equations is integrable. The first goal is obtained  if
\begin{equation} \label{aa}
 -\frac{3}{8}+\frac{3}{2} a=0 \qquad \Longrightarrow \qquad a=\frac{1}{4}. 
\end{equation}

Integrability is achieved when coefficient in front of the term with $\eta_{x}\eta_{2x}$ is twice the coefficient in front of the term with $\eta\eta_{3x}$. So, we can choose  $b$ such that
\begin{equation} \label{bb}
\frac{23}{24} +a- 3b = 2\frac{5}{12}  \qquad \Longrightarrow \qquad b= \frac{1}{8}.
\end{equation}
Then, applying to (\ref{kdv2'}) NIT (\ref{nit}) with parameters $a=\frac{1}{4}$ and $b=\frac{1}{8}$ we obtain asymptotically equivalent integrable equation in the form
\begin{equation} \label{kdv2a}
\eta'_{t}   + \frac{3}{2} \alpha\,\eta'\eta'_{x}+ \frac{1}{6}\beta\, \eta'_{3x} +\frac{5}{12}
  \alpha\beta\,\left(2\eta'_{x}\eta'_{2x}+\eta'\eta'_{3x} \right) =0 .  
\end{equation}


We will show that for  (\ref{kdv2a}) there exists Hamiltonian form
\begin{equation} \label{h4}
\eta'_t = \frac{\partial}{\partial x} \left(\frac{\delta \mathcal{H}}{\delta \eta'} \right),
\end{equation}
where Hamiltonian ~$H=\int_{-\infty}^{\infty} \mathcal{H}\,dx$~ has the density
\begin{equation} \label{h5}
\mathcal{H} = -\frac{1}{4}\alpha \eta'^{\,3}+ \frac{1}{12}\beta \eta'^{\,2}_x + \frac{5}{24} \alpha\beta \eta'\eta'^{\,2}_x.
\end{equation}

Since ~$\mathcal{H}=\mathcal{H}(\eta',\eta'_x)$, then functional derivative is given by
\begin{equation} \label{h6}
\frac{\delta \mathcal{H}}{\delta \eta'} = \frac{\partial \mathcal{H}}{\partial \eta'} - \frac{\partial}{\partial x}\frac{\partial \mathcal{H}}{\partial \eta'_x} 
= -\frac{3}{4} \alpha  \eta'^{\,2} -\frac{1}{6} \beta 
   \eta'_{2x} - \frac{5}{24} \alpha  \beta 
   {\eta'_{x}}^{2} - \frac{5}{12} \alpha  \beta \eta' \eta'_{2x}.
\end{equation}
Insertion of (\ref{h6}) into  (\ref{h4}) gives
\begin{equation} \label{h7}
\eta'_t = -\frac{3}{2} \alpha \eta' \eta'_x -\frac{1}{6} \beta \eta'_{3x} - \frac{5}{12} \alpha\beta \left(2 \eta'_x \eta'_{xx} + \eta' \eta'_{xxx}
\right).
\end{equation}
what coincides with (\ref{kdv2a}).

It is worth to notice, that application of inverse NIT to (\ref{kdv2a}) brings back the equation (\ref{kdv2'}) (up to second order in $\alpha,\beta$).

Existence of the Hamiltonian implies that there exist invariats of the equation (\ref{kdv2a}). This is the first step  towards obtaining a global mild solution according to approach due to de Bouard and Debussche \cite{Deb}.

\begin{remark}
Equations (\ref{kdv2a}) or (\ref{h7}), up to numerical coefficients, are the same as left hand side of  stochastic equation (\ref{NL1}). Then study of  stochastic equation  (\ref{NL1}) is justified.
\end{remark}

\section{Proof of Theorem \ref{aux}} \label{proofs3}

In order to make the paper self-contained we recall the following results.

\begin{thm}\label{A1}(\textit{\cite{Deb}, Proposition A.1})
Let $A=L_{\omega}^{q}(L_{t}^{2})$ or $A=L^{q}(\Omega)$, with $1<q<\infty$, and let $u$ be an $A$-valued function of $x\in\mathbb{R}$. Assume that for some $p$, with $1<p<\infty$ and some $\sigma>0$ 
$$	u\in L_{x}^{p} ( A), \quad D^{\sigma}u\in L_{x}^{\infty}(A);$$
then for any $\alpha \in [0,\sigma]$ 
$D^{\alpha}u \in L_{x}^{p_{\alpha}}$, with ~$p_{\alpha}$ defined by ~$\frac{1}{p_{\alpha}} = \frac{1}{p}\left( 1- \frac{\alpha}{\sigma} \right)$.
Furthermore, there is a constant $C$ such that $$
\left| D^{\alpha} \right|_{L_{x}^{p_{\alpha}}(A)} \leq C\left| u \right|_{L_{x}^{p_{\alpha}}(A)}^{1-\frac{\alpha}{\sigma}}\left|D^{\sigma}u\right|_{L_{x}^{\infty}(A)}^{\frac{\alpha}{\sigma}}.$$
\end{thm}

\begin{thm}\label{L2.1}(\textit{\cite{KPV91}, Lemma 2.1})
Let $v_{0}\in L^{2}(\mathbb{R})$.Then
\begin{equation}\nonumber
\int_{-\infty}^{\infty} \left| D^{\frac{\alpha}{2}}V^{\alpha}(t)v_{0}(x)\right|^{2} \diff t = c_{\alpha}||v_{0}||^{2}_{2} \quad \mbox{for~any} \quad x\in\mathbb{R}.
\end{equation}
\end{thm}

\begin{thm}\label{T2.4}(\textit{\cite{KPV91}, Theorem 2.4})
For any $(\theta,\beta)\in[0,1]\times\left[0,\frac{\alpha-1}{2}\right]$ 
\begin{equation}\label{T24a}
\left(\int_{-\infty}^{\infty} \left\|D^{\theta\frac{\beta}{2}}U^{\alpha}(t)v_{0}\right\|^{q}_{p}\diff t\right)^{\frac{1}{q}} \leq c\left\|v_{0}\right\|_{2}
\end{equation}
and
\begin{equation}\label{T24b}
\left(\int_{-\infty}^{\infty} \left\|\int D^{\theta\frac{\beta}{2}}U^{\alpha}(t-s)f(\cdot,s)\diff s\right\|^{q}_{p}\diff t\right)^{\frac{1}{q}} \leq c\left(\int_{-\infty}^{\infty}\left\|f(\cdot,s)\right\|_{p'}^{q'} \diff t\right)^{\frac{1}{q'}},
\end{equation}
where $(q,p) = \left(2(\alpha+1)/(\theta(\beta+1)),2/(1-\theta)\right)$, $\frac{1}{p}+\frac{1}{p'} = \frac{1}{q}+\frac{1}{q'} = 1$.
\end{thm}

\begin{lemma}\label{P3.3}
Asume that $\widetilde{\sigma}>\sigma>\frac{3}{4}$ and $0<\varepsilon<\inf\left\{\widetilde{\sigma},2\right\}$. Then 
\begin{equation}\nonumber
D^{\widetilde{\sigma}-\varepsilon}\partial_{x}\left(\frac{\partial^{2}}{\partial x^{2}}W_{V} \right)\in L^{2}\left(\Omega; L_{x}^{\infty}(L_{t}^{2})\right).
\end{equation}

\end{lemma}
\begin{proof}
Let, as usually,  $W_{V} := \int_{0}^{t} V(t-s)\Phi\diff W(s)$ and let $q=\frac{6}{\varepsilon}$.
Estimate the expression $\left|D^{3+\widetilde{\sigma}}W_{V}\right|_{L_{x}^{\infty}(L_{\omega}^{q}(L_{t}^{2}))}$. We have
\begin{equation}\label{P33E1}
\begin{aligned}
\left|D^{3+\widetilde{\sigma}}W_{V}\right|_{L_{x}^{\infty}(L_{\omega}^{q}(L_{t}^{2}))} =& \sup_{x\in\mathbb{R}}\mathbb{E}\left(\left( \int_{0}^{T}\left|\int_{0}^{t} D^{3+\widetilde{\sigma}} V(t-s) \Phi \diff W(s) \right| ^{2} \diff t \right) ^{\frac{q}{2}} \right) \leq \\
\leq& C \sup_{x\in\mathbb{R}} \int_{0}^{T} \mathbb{E} \left( \left| \int_{0}^{t} D^{3+\widetilde{\sigma}} V(t-s) \Phi \diff W(s) \right| ^{2} \right)^{\frac{q}{2}} \diff t \leq \\
\leq& C \sup_{x\in\mathbb{R}} \int_{0}^{T}  \left(   \int_{0}^{t} \sum_{i\in\mathbb{N}} \left| D^{3+\widetilde{\sigma}} V(t-s) \Phi e_{i}(s) \right| ^{2} \diff s \right)^{\frac{q}{2}} \diff t \leq \\
\leq& C  \int_{0}^{T}   \left( \sum_{i\in\mathbb{N}} \sup_{x\in\mathbb{R}} \int_{0}^{t}  \left| D^{3+\widetilde{\sigma}} V(t-s) \Phi e_{i}(s) \right| ^{2} \diff s \right)^{\frac{q}{2}} \diff t.
\end{aligned}
\end{equation}
 
Let us substitute in Theorem \ref{L2.1} $v_{0} = D^{\widehat{\sigma}+\frac{5}{2}}\Phi e_{i}$~ ~and ~$\alpha = 1$. Then we obtain
\begin{equation}\nonumber
\begin{aligned}
C\left|D^{\widehat{\sigma}+\frac{5}{2}}\Phi e_{i}\right|_{L^{2}}^{2} =& \int_{-\infty}^{\infty}|D^{\frac{1}{2}}V(s)D^{\widehat{\sigma}+\frac{5}{2}}\Phi e_{i}|^{2} \diff s = \int_{-\infty}^{\infty}|D^{3+\widehat{\sigma}}V(s)\Phi e_{i}|^{2} \diff s \geq \int_{0}^{t}|D^{3+\widehat{\sigma}}V(t-s)\Phi e_{i}|^{2} \diff s.
\end{aligned}
\end{equation}

Since Theorem \ref{L2.1} holds for all $x\in\mathbb{R}$ and $\left|D^{\widehat{\sigma}+\frac{5}{2}}\Phi e_{i}\right|_{L^{2}} \leq \left|\Phi e_{i}\right|_{H^{\widehat{\sigma}+\frac{5}{2}}}$, then
\begin{equation}\label{P33E2}
\sup_{x\in\mathbb{R}} \int_{0}^{t}|D^{3+\widehat{\sigma}}V(t-s)\Phi e_{i}|^{2} \diff s \leq C\left|\Phi e_{i}\right|_{H^{\widehat{\sigma}+\frac{5}{2}}}^{2}.
\end{equation}

Insertion of (\ref{P33E2}) into (\ref{P33E1}), gives
\begin{equation}\label{P33E3}
\left|D^{3+\widetilde{\sigma}}W_{V}\right|^{q}_{L_{x}^{\infty}(L_{\omega}^{q}(L_{t}^{2}))} \leq C\int_{0}^{T}\left(\sum_{i\in\mathbb{N}} \left|\Phi e_{i}\right|_{H^{\widetilde{\sigma}+\frac{5}{2}}}^{2}\right)^{\frac{q}{2}} \diff t \leq C(T)\left| \Phi\right|_{2}^{0,\widehat{\sigma}+\frac{5}{2}} \leq C(T)|\Phi|^{q}_{L_{2}^{0,\widetilde{\sigma}+\frac{5}{2}}}.
\end{equation}

Let us estimate  $\left|D^{\widetilde{\sigma}}W_{V}\right|_{L_{x}^{2}(L_{\omega}^{q}(L_{t}^{2}))}^{2}$. Basing on proof of Proposition 3.3 in de Bouard and Debussche \cite{Deb} we  have that
\begin{equation}\label{P33E4}
\begin{aligned}
\left|D^{\widetilde{\sigma}}W_{V}\right|_{L_{x}^{2}(L_{\omega}^{q}(L_{t}^{2}))}^{2} \leq C|\Phi|_{L_{2}^{0,\widetilde{\sigma}}}^{2} \leq C|\Phi|^{q}_{L_{2}^{0,\widetilde{\sigma}+\frac{5}{2}}}.
\end{aligned}
\end{equation}

Now, set in Theorem  \ref{A1} $A=L_{\omega}^{q}(L_{t}^{2})$, $p=2$, $u=D^{\widetilde{\sigma}}W_{V}$ $\sigma = 3$ and $\alpha = 3 - \varepsilon$ for some $3 > \varepsilon >0$. Then $D^{3 - \varepsilon} D^{\widetilde{\sigma}} W_{V} = D^{3+\widetilde{\sigma}- \varepsilon}W_{V} \in L_{x}^{p_{\alpha}}(L_{\omega}^{q}(L_{t}^{2}))$ and there exists a constant $C$, such that
\begin{equation}
\begin{aligned}
\left|D^{3+\widetilde{\sigma} - \varepsilon} W_{V}\right|_{L_{x}^{p_{\alpha}}(L_{\omega}^{q}(L_{t}^{2}))} &\leq C \left|D^{\widetilde{\sigma}}W_{V}\right|_{L^{2}(L_{\omega}^{q}(L_{t}^{2}))}^{1-\frac{3-\varepsilon}{3}}\left|D^{3+\widetilde{\sigma}}W_{V}\right|_{L_{x}^{\infty}(L_{\omega}^{q}(L_{t}^{2}))}^{\frac{3-\varepsilon}{3}} = \\
 &= C \left|D^{\widetilde{\sigma}}W_{V}\right|_{L^{2}(L_{\omega}^{q}(L_{t}^{2}))}^{\frac{\varepsilon}{3}}\left|D^{3+\widetilde{\sigma}}W_{V}\right|_{L_{x}^{\infty}(L_{\omega}^{q}(L_{t}^{2}))}^{1-\frac{\varepsilon}{3}},
\end{aligned}
\end{equation}
where 
\begin{equation}\nonumber
p_{\alpha} = \left(\frac{1}{2}\left(1-\frac{3 - \varepsilon}{3}\right)\right)^{-1} = \left(\frac{- \frac{\varepsilon}{3}}{2}\right)^{-1} = \frac{6}{\varepsilon} = q.
\end{equation}

Then we have
\begin{equation}\label{P33E5}
\begin{aligned}
\left|D^{3+\widetilde{\sigma} - \varepsilon} W_{V}\right|_{L_{x}^{q}(L_{\omega}^{q}(L_{t}^{2}))} &\leq C \left|D^{\widetilde{\sigma}}W_{V}\right|_{L^{2}(L_{\omega}^{q}(L_{t}^{2}))}^{\frac{2}{q}}\left|D^{3+\widetilde{\sigma}}W_{V}\right|_{L_{x}^{\infty}(L_{\omega}^{q}(L_{t}^{2}))}^{1-\frac{2}{q}} \leq C\left|\Phi\right|_{L_{2}^{0,\widetilde{\sigma}+\frac{5}{2}}}.
\end{aligned}
\end{equation}
Moreover, basing on the proof of Proposition 3.3 in de Bouard and Debussche  \cite{Deb}, 
\begin{equation}\label{P33E6}
\left| W_{V} \right|_{L_{\omega}^{q}(L_{x}^{q}(L_{t}^{2}))} \leq C\left|\Phi\right|_{L_{2}^{0,\widetilde{\sigma}}} \leq C\left|\Phi\right|_{L_{2}^{0,\widetilde{\sigma}+\frac{5}{2}}}.
\end{equation}

Since $\left|D^{3+\widetilde{\sigma} - \varepsilon} W_{V}\right|_{L_{\omega}^{q}(L_{x}^{q}(L_{t}^{2}))} = \left|D^{3+\widetilde{\sigma} - \varepsilon} W_{V}\right|_{L_{x}^{q}(L_{\omega}^{q}(L_{t}^{2}))}$, then from (\ref{P33E5}) oraz (\ref{P33E6})  we obtain
\begin{equation}\label{P33E7}
\left|W_{V}\right|_{L_{\omega}^{q}(W_{x}^{3+\widetilde{\sigma}-\varepsilon,q}(L_{t}^{2}))}\leq C\left|\Phi\right|_{L_{2}^{0,\widetilde{\sigma}+\frac{5}{2}}}.
\end{equation}
Because $q\,\varepsilon>1$, then $W^{\varepsilon,q}_{x}(L_{t}^{2})\subset L_{x}^{\infty}(L_{t}^{2})$, therefore $D^{3+\widetilde{\sigma} - \varepsilon} W_{V} \in L_{\omega}^{q}(L_{x}^{\infty}(T_{t}^{2}))$. Moreover 
\begin{equation}\nonumber
\left|D^{3+\widetilde{\sigma} - \varepsilon} W_{V}\right|_{L_{\omega}^{q}(L_{x}^{\infty}(L_{t}^{2}))} \leq C\left|\Phi\right|_{L_{2}^{0,\widetilde{\sigma}+\frac{5}{2}}}.
\end{equation}

Finally 
\begin{eqnarray*}
D^{\widetilde{\sigma}-\varepsilon}\partial_{x}\left(\frac{\partial^{2}}{\partial x^{2}}W_V\right) &=& D^{\widetilde{\sigma}-\varepsilon}\partial_{3x}W_V = \int_{0}^{t}D^{\widetilde{\sigma}-\varepsilon}\partial_{3x}V(t-s)\Phi\diff W(s)\\ &=&  \int_{0}^{t}D^{3+\widetilde{\sigma}-\varepsilon}V(t-s)\mathcal{H}\Phi\diff W(s),
\end{eqnarray*}
where $\mathcal{H}$ is the Hilbert transform, what finishes the proof.
\end{proof}

\begin{lemma}\label{P3.4}
\begin{equation}\nonumber
\partial_{x}\left(\frac{\partial^{2}}{\partial x^{2}}W_{V}\right) \in L^{2}\left(\Omega;L_{t}^{4}\left(L_{x}^{\infty}\right)\right).
\end{equation}
\end{lemma}
\begin{proof}
Let $\varepsilon = \widetilde{\sigma} - \frac{3}{4}$ and $q=4+\frac{12}{\varepsilon}$.
Estimate  ~$\left| D^{3+\varepsilon} W_{V}\right|_{L_{t}^{4}(L_{x}^{\infty}(L_{\omega}^{q}))}$. 

We have
\begin{equation}\nonumber
\begin{aligned}
\left| D^{3+\varepsilon} W_{V}\right|^{4}_{L_{t}^{4}(L_{x}^{\infty}(L_{\omega}^{q}))} =& \int_{0}^{T}\sup_{x\in\mathbb{R}}\mathbb{E}\left(\left|\int_{0}^{t}D^{\widetilde{\sigma}+\frac{9}{4}}V(t-s)\Phi\diff W(s)\right|^{q}\right)^{\frac{4}{q}} \diff t \leq \\
\leq& C \int_{0}^{T}\sup_{x\in\mathbb{R}}\left(\sum_{i\in\mathbb{N}}\int_{0}^{t}\left|D^{\widetilde{\sigma}+\frac{9}{4}}V(t-s)\Phi e_{i}\right|^{2}\diff s\right)^{2}\diff t \leq \\
\leq& C(T) \left( \sum_{i\in\mathbb{N}}\left( \int_{0}^{T} \sup_{x\in\mathbb{R}} \left|D^{\widetilde{\sigma}+\frac{9}{4}}V(t-s)\Phi e_{i}\right|^{4} \diff s\right)^{\frac{1}{2}}\right)^{2}.
\end{aligned}
\end{equation}
Substitute in Theorem  \ref{T2.4} $\alpha=2$, $\theta=1$, $\beta=\frac{1}{2}$ (like in de Bouard and Debussche \cite{Deb}). The result is
\begin{equation}\nonumber
\int_{0}^{T}\sup_{x\in\mathbb{R}}\left|D^{\widetilde{\sigma} + \frac{9}{4}}V(t)\Phi e_{i}\right|^{4}\diff t \leq C\left|D^{\widetilde{\sigma}+2}\Phi e_{i}\right|^{4}_{L_{x}^{2}}\leq C\left|\Phi\right|_{L_{2}^{0,\widetilde{\sigma}+2}},
\end{equation}
where ~$L_{2}^{0,\widetilde{\sigma}+2} = L_2^0 (L^2(\mathbb{R});H^{\widetilde{\sigma}+2}(\mathbb{R}) )$.
This implies 
\begin{equation}\nonumber
\left| D^{3+\varepsilon} W_{V}\right|^{4}_{L_{t}^{4}(L_{x}^{\infty}(L_{\omega}^{q}))} \leq C\left|\Phi\right|_{L_{2}^{0,\widetilde{\sigma}+2}},
\end{equation}
Moreover from the proof of Proposition 3.4 in de Bouard and Debussche  \cite{Deb} we know that
\begin{equation}\nonumber
\left|W_{V}\right|_{L_{t}^{4}(L_{x}^{2}(L_{\omega}^{q}))} \leq C\left|\Phi\right|_{L_{2}^{0,0}} \leq C\left|\Phi\right|_{L_{2}^{0,\widetilde{\sigma}}} \leq C\left|\Phi\right|_{L_{2}^{0,\widetilde{\sigma}+2}}.
\end{equation}

Now substitute in Theorem \ref{A1} $\sigma = 3+\varepsilon$, $A=L_{\omega}^{q}=L_{q}(\Omega)$, $p=2$, $u=W_{V}$, $\alpha=3+\frac{\varepsilon}{2}$. Then $p_{\alpha}=\left(\frac{1}{2}\left(1-\frac{3+\frac{\varepsilon}{2}}{3+\varepsilon}\right)\right)^{-1} = 4+\frac{12}{\varepsilon}=q$~ and
\begin{equation}\nonumber
\begin{aligned}
\left|D^{3+\frac{\varepsilon}{2}}W_{V}\right|_{L_{x}^{p_{\alpha}}(L_{\omega}^{q})}\leq& C\left|W_{V}\right|^{1-\frac{3+\frac{\varepsilon}{2}}{3+\varepsilon}}_{L_{x}^{p_{\alpha}}(L_{\omega}^{q})}\left|D^{3+\varepsilon}W_{V}\right|_{L_{x}^{\infty}(L_{\omega}^{q})}^{\frac{3+\frac{\varepsilon}{2}}{3+\varepsilon}} = C\left|W_{V}\right|^{\frac{\frac{\varepsilon}{2}}{3+\varepsilon}}_{L_{x}^{p_{\alpha}}(L_{\omega}^{q})}\left|D^{3+\varepsilon}W_{V}\right|_{L_{x}^{\infty}(L_{\omega}^{q})}^{\frac{3+\frac{\varepsilon}{2}}{3+\varepsilon}} = \\
=& C\left|W_{V}\right|^{\frac{2}{q}}_{L_{x}^{q}(L_{\omega}^{q})}\left|D^{3+\varepsilon}W_{V}\right|_{L_{x}^{\infty}(L_{\omega}^{q})}^{1-\frac{2}{q}}.
\end{aligned}
\end{equation}

Since $q=4+\frac{12}{\varepsilon}\geq 4$, then
\begin{equation}\nonumber
\begin{aligned}
\left|D^{3+\frac{\varepsilon}{2}}W_{V}\right|_{L_{\omega}^{4}(L_{t}^{4}(L_{x}^{q}))} \leq & \left|D^{3+\frac{\varepsilon}{2}}W_{V}\right|_{L_{t}^{4}(L_{x}^{q}(L_{\omega}^{q}))} \leq C\left|W_{V}\right|^{\frac{2}{q}}_{L_{x}^{q}(L_{\omega}^{q})}\left|D^{3+\varepsilon}W_{V}\right|_{L_{x}^{\infty}(L_{\omega}^{q})}^{1-\frac{2}{q}} \leq \\ 
\leq & C\left| \Phi \right|_{L_{2}^{0,\widetilde{\sigma}}} \leq C\left| \Phi \right|_{L_{2}^{0,\widetilde{\sigma}+ 2}}.
\end{aligned}
\end{equation}

The proof of Proposition 3.4 in de Bouard and Debussche \cite{Deb} implies that
\begin{equation}\nonumber
\left|W_{V}\right|\leq C(T)\left|\Phi\right|_{L_{2}^{0,\widetilde{\sigma}}} \leq C(T)\left|\Phi\right|_{L_{2}^{0,\widetilde{\sigma}+2}},
\end{equation}
therefore
\begin{equation}\nonumber
\left|W_{V}\right|_{L_{\omega}^{4}\left(L_{t}^{4}\left(W_{x}^{3+\frac{\varepsilon}{2},q}\right)\right)} \leq C\left|\Phi\right|_{L_{2}^{0,\widetilde{\sigma}+2}}
\end{equation}
and, since  $q\,\varepsilon/2>1$, 
\begin{equation}\nonumber
\left|\partial_{3x} W_{V}\right|_{L_{\omega}^{4}(L_{t}^{4}(L_{x}^{\infty}))} \leq C\left|W_{V}\right|_{L_{\omega}^{4}\left(L_{t}^{4}\left(W_{x}^{3+\frac{\varepsilon}{2},q}\right)\right)} \leq C\left|\Phi\right|_{L_{2}^{0,\widetilde{\sigma}+2}}.
\end{equation}
\end{proof}


\begin{thebibliography}{99}
\bibitem{Adams} Adams, R.A.  \emph{Sobolev Spaces}, Academic Press: New York,  1975.

\bibitem{Deb} de Bouard, A.; Debussche A. On the stochastic Korteweg - de Vries Equation. J. Funct. Anal. \textbf{1998}, \textit{54}, 215-251.

\bibitem{DrazJohn} Drazin, P.G.;  Johnson, R.S.
\emph{Solitons: An introduction}, Cambridge University Press, UK, 1989.

\bibitem{Dull2001} Dullin, H.R.; Gottwald, G.A.; Holm, D.D. An integrable shallow water equation with linear and nonlinear dispersion. Phys.\ Rev.\ Lett.\  \textbf{2001}, \textit{87} (19), 194501.

\bibitem{Grimshaw}
Grimshaw, R.  INTERNAL SOLITARY WAVES, Presented at the international conference ``PROGRESS IN NONLINEAR SCIENCE", held in
Nizhni Novgorod in July 2001, and dedicated to the 100-th Anniversary of Alexander A. Andronov.

\bibitem{GrimLect}
Grimshaw, G.; El, G.; and Khusnutdinova, K.  Nonlinear Waves, Lecture 12: Higher-order KdV equations. 2010.

\bibitem{EIGR} Infeld, E.; Rowlands, G. \textit{Nonlinear Waves, Solitons and Chaos}, Cambridge University Press, 2nd Edition: UK, 2000.

\bibitem{IKRR} Infeld, E.; Karczewska, A.; Rozmej, P.;  Rowlands, G.
Exact solitonic and periodic solutions of the extended KdV equation.  submitted. https://arxiv.org/pdf/1612.03847.pdf

\bibitem{Jeffrey}
Jeffrey, A.. Role of the Korteweg-de Vries equation in plasma physics. 
Quarterly Journal of the Royal Astronomical Society, \textbf{1973},  \textit{14}, 183-189. 

\bibitem{KRR} Karczewska, A.; Rozmej, P.; Rutkowski, \L{}.
A new nonlinear equation in the shallow water wave problem.
Physica Scripta, \textbf{2014}, \textit{89}, 054026.

\bibitem{KRI} Karczewska, A.; Rozmej, P.; and Infeld, E.
Shallow water soliton dynamics beyond KdV.
Phys. Rev. E, \textbf{2014}, \textit{90},  012907.

\bibitem{KRI2}  Karczewska, A.; Rozmej, P.;  Infeld, E. Energy invariant for shallow-water waves and the Korteweg - de Vries equation:
Doubts about the invariance of energy. Phys. Rev. E, \textbf{2014}, \textit{92}, 053202.  

\bibitem{KRIad} Karczewska, A.; Rozmej, P.; Infeld, E.; Rowlands, G. 
Adiabatic invariants of the extended KdV equation. Phys. Lett. A, \textbf{2017}, \textit{381} (4), 270-275.

\bibitem{KPV91} Kenig, C.E.; Ponce ~G.; Vega, L. Well-posedness of the initial value problem for the Korteweg-de Vries equation. J. Amer. Math. Soc. \textbf{1991}, \textit{4}, 323-347. 

\bibitem{KPV93} Kenig, C.E.; Ponce ~G.; Vega, L. Well-posedness 
for the generalized  Korteweg-de Vries equation via contraction principle. Comm. Pure Appl. Math.   \textbf{1993}, \textit{46}, 527-620. 

\bibitem{Kodama} 
Kodama, Y.  On integrable systems with higher order corrections. 
Phys. Lett. A. \textbf{1985}, {\it 107}, 245-249.

\bibitem{KdV} Korteweg, D.J.;  de Vries, H., 
On the change of form of long waves advancing in a rectangular canal, and on a new type of long stationary waves. 
Philosophical Magazine. \textbf{1985}, \textit{39}, 422-443. 

\bibitem{LiPo} Linares, F.;  Ponce, G.,  \emph{Introduction to Nonlinear Dispersive Equations}. Universitext, Springer: Germany, 2009.

\bibitem{MS90} Marchant, T.R.;  Smyth, N.F.  The extended Korteweg--de Vries equation and the resonant flow of a fluid over topography.  J.\ Fluid Mech.  \textbf{1990}, {\it 221}, 263-288.

\bibitem{Rem} Remoissenet, M.
\emph{Waves called solitons}, Springer-Verlag: Germany 1994.

\bibitem{Tao} Tao, T.  \emph{Nonlinear Dispersive Equations, Local and Global Analysis}, CBMS Regional Conference Series, 106, American Mathematical Society: USA, 2006.

\end{thebibliography}
\end{document}